\documentclass[a4paper, 11pt]{article}
\RequirePackage{fullpage}

\usepackage{multicol}
\usepackage{amsthm,amsmath,amssymb}
\usepackage{booktabs,array}
\usepackage{subcaption}
\usepackage[usenames,dvipsnames]{xcolor}
\usepackage[pdftex,breaklinks,colorlinks,
    citecolor={BlueViolet}, linkcolor={Blue},urlcolor=Maroon]{hyperref}

\usepackage{charter,eulervm}
\usepackage{enumerate}
\usepackage[final,expansion=alltext,protrusion=true]{microtype}

\theoremstyle{plain} 
\newtheorem{theorem}{Theorem}[section]
\newtheorem{lemma}[theorem]{Lemma}
\newtheorem{corollary}[theorem]{Corollary}
\newtheorem{proposition}[theorem]{Proposition}
\newtheorem{conjecture}[theorem]{Conjecture}

\title{On Well (Edge) Dominated and Equimatchable \\Strong Product Graphs}
\author{
  Yixin Cao\thanks{School of Computer Science and Engineering, Central South University, Changsha, China. { guiqiang.mou@csu.edu.cn, jxwang@mail.csu.edu.cn}} \thanks{Department of Computing, Hong Kong Polytechnic University, Hong Kong, China. { yixin.cao@polyu.edu.hk}} 
  \and Guiqiang Mou\footnotemark[1]
  \and Jianxin Wang\footnotemark[1]
}
\date{}

\begin{document}
\maketitle
\begin{abstract}
  A graph is well-(edge-)dominated if every minimal (edge) dominating set is minimum. A graph is equimatchable if every maximal matching is maximum. We study these concepts on strong product graphs. We fully characterize well-edge-dominated and equimatchable strong product graphs of nontrivial graphs, and identify a large family of graphs whose strong products with any well-dominated graph are well-dominated.
  
\end{abstract}

\section{Introduction}

All the graphs discussed in this paper are finite and simple.
A vertex subset~$D$ of a graph $G$ is a \emph{dominating set} of $G$ if any vertex not in~$D$ has a neighbor in~$D$.
Generally, a minimal dominating set can be arbitrarily larger than a minimum one, e.g., the two minimal dominating sets of a star graph.
A graph is \emph{well-dominated} if all its minimal dominating sets are minimum.
The concept of well-dominated graphs was motivated by the well-covered graphs defined by Plummer~\cite{plummer1970some}.
A graph is \emph{well-covered} if every maximal independent set (a set of pairwise nonadjacent vertices) is maximum.
Since every maximal independent set is a minimal dominating set, a well-dominated graph is well-covered.  The other direction does not hold; e.g., $K_{4, 4}$ is well-covered but not well-dominated.
These two concepts are equivalent on graphs with no cycles of length 4 or 5~\cite{levit2017well}.

Both concepts have natural edge correspondences.  
We say that two edges are \emph{adjacent} if they share an endpoint.
An edge subset $F$ of a graph $G$ is an \emph{edge dominating set} of $G$ if any edge not in $F$ is adjacent to at least one edge in $F$.
A \emph{matching} is a set of edges that are pairwise nonadjacent.
Note that there is a one-to-one mapping between edge dominating sets (resp., matchings) of a graph $G$ and dominating sets (resp., independent sets) of the line graph of~$G$.
A graph is \emph{well-edge-dominated} if all its minimal edge dominating sets are minimum, and a graph is \emph{equimatchable} if all its maximal matchings are maximum~\cite{frendrup2009note}.
Any maximal matching of a graph is also a minimal edge dominating set, and hence a well-edge-dominated graph is equimatchable.
The other direction is not true; e.g., $K_{2,3}$ is equimatchable but not well-edge-dominated.

There have been efforts in characterizing well-dominated graphs, well-edge-dominated graphs, and equimatchable graphs.  For earlier results, we refer to the survey of Plummer~\cite{plummer1993well}.
A recent line of study is to characterize product graphs with these properties.
One particular motivation for it, especially on well-dominated ones, is the widely open Vizing's conjecture~\cite{vizing1968some}, which is concerned with domination of Cartesian product graphs, and has also motivated the study of domination in other product graphs.
In particular, well-dominated graphs have been fully characterized within Cartesian products~\cite{kuenzel2024characterization}, lexicographic products~\cite{gozupek2017characterizations}, disjunctive products~\cite{anderson2021well}, and direct products~\cite{rall2023well}, 
while well-edge-dominated graphs have only been characterized within Cartesian product graphs~\cite{anderson2022well}.
Equimatchable graphs have been characterized when they are perfect-matchable~\cite{summer1979randomly}, claw-free~\cite{akbari2018equimatchable}, or triangle-free~\cite{buyukccolak2022triangle}.  As far as we know, there was no previous study of equimatchable product graphs.

This work is focused on strong product graphs that are well-dominated, well-edge-dominated, or equimatchable.  
In the \emph{strong product} of two graphs $G$ and $H$, denoted as $G\boxtimes H$, the vertex set is $V(G)\times V(H)$, and two distinct vertices $(u_1, v_1)$ and $(u_2, v_2)$ are adjacent if and only if $u_{2}\in N_G[u_{1}]$ and $v_{2}\in N_H[v_{1}].$
In other words, one of the following three conditions holds true: (i)~$u_1u_2\in E(G)$ and $v_1=v_2$, (ii)~$v_1v_2\in E(H)$ and $u_1=u_2$, or (iii)~$u_1u_2\in E(G)$ and $v_1v_2\in E(H)$.
The graphs $G$ and~$H$ are the \emph{factors} of $G\boxtimes H$.
Note that any graph is a strong product of itself and a trivial graph (on one vertex).  
In this case, the problem reduces to whether the nontrivial factor has the desired property.  Thus, throughout this work, we assume that both factors are nontrivial.

Rall~\cite{rall2023well} showed that for a strong product graph to be well-dominated, it is necessary that both factors are well-dominated.  This condition is sufficient when one of the factors is a \emph{complete} graph (a graph with all possible edges present), but not sufficient in general~\cite{rall2023well}.
We say that a graph~$G$ is \emph{trivially well-dominated} if there is a set of vertices~$v_{1}$, $v_{2}$, $\ldots$, $v_{t}$ such that their closed neighborhoods
is a clique partition of~$V(G)$; i.e., each closed neighborhood is a clique and every vertex is in precisely one of them.
For the definition, note that any minimal dominating set of~$G$ must contain exactly one vertex from each of these cliques.  Thus, $t$ is the domination number of~$G$.

\begin{theorem}\label{thm:well-dominated}
  Let $G$ be a trivially well-dominated graph.  For any well-dominated graph $H$, the strong product $G\boxtimes H$ is well-dominated.
\end{theorem}

Note that a complete graph is trivially well-dominated, with $t = 1$.  Thus, Theorem~\ref{thm:well-dominated} subsumes a result of Rall~\cite{rall2023well}, in which~$G$ needs to be a complete graph.
Moreover, all well-dominated chordal graphs are trivially well-dominated~\cite{prisner1996well}.
On the other hand, we conjecture that no other graphs have this property.  In other words, for each well-dominated graph that is not trivially well-dominated, there exists another well-dominated graph such that their strong product graph is not well-dominated.

For well-edge-dominated graphs, we give a full characterization within strong product graphs.
We use strong induction on the sum of matching numbers of the factors to show that at least one factor has a \emph{perfect matching} (a matching whose endpoints involve the whole vertex set) for well-edge-dominated strong product graphs.
As said, a perfect matching is a minimal edge dominating set, which means that such a graph has a very simple structure~\cite{anderson2022well}.
Let us remark that our proof can be used to prove a similar result of Anderson et al.~\cite{anderson2022well} on Cartesian product graphs, but not the other way.

\begin{theorem}\label{thm:well-edge-dominated}
  A strong product of two nontrivial and connected graphs is well-edge-dominated if and only if both factors consist of a single edge.
\end{theorem}

As said, every well-edge-dominated graph is equimatchable.  Anderson et al.~\cite{anderson2022well} showed that these two concepts are equivalent on Cartesian product graphs when both factors are nontrivial and connected.  However, they are not equivalent within strong product graphs. 
Our third result is a full characterization of equimatchable strong product graphs whose factors are nontrivial and connected. 

It is easy to check by definition that all complete graphs are equimatchable.
Let~$\alpha(G)$ denote the \emph{independence number} of graph~$G$, i.e., the cardinality of a maximum independent set of~$G$.
Note that a graph~$G$ is complete if and only if~$\alpha(G) = 1$.
Akbari et al.~\cite{akbari2018equimatchable} observed that for odd graphs, the condition can be relaxed: an odd connected graph~$G$ is equimatchable if~$\alpha(G) \le 2$.
Here by an \emph{odd (resp., even) graph} we mean a graph of an odd (resp., even) order.
Note that the bound of~$\alpha(G) \le 2$ is tight because the independence number of the bull graph (the graph obtained from~$P_{5}$ by adding an edge between the second and fourth vertices), which is not equimatchable, is three.
It turns out that whether a strong product graph is equimatchable can be fully decided by the independence numbers of the factors.

\begin{theorem}\label{thm:equimatchable}
  Let~$G$ and~$H$ be two nontrivial and connected graphs.
  The strong product $G\boxtimes H$ is equimatchable if and only if 
  \begin{description}
  \item[Even] 
    $\alpha(G) + \alpha(H) = 2$
    when one of~$G$ and~$H$ is even; or
  \item[Odd] $\alpha(G) + \alpha(H)\leq 3$ when both~$G$ and~$H$ are odd.
  \end{description}
\end{theorem}

A graph is well-dominated, well-edge-dominated, or equimatchable if and only if all its components are.  A strong product is connected if and only if both factors are connected.  Therefore, throughout the paper we focus on connected graphs.

\section{Well-dominated strong product graphs}

The set of \emph{neighbors} of a vertex~$v$ in~$G$ is denoted as~$N_{G}(v)$, and $N_{G}[v] = N_{G}(v)\cup \{v\}$.
We use $\gamma(G)$ to denote the cardinality of a smallest dominating set of~$G$. 

\begin{proof}[Proof of Theorem~\ref{thm:well-dominated}]
  Let $u_{1}$, $u_{2}$, $\ldots$, $u_{t}$ be a set of vertices such that $\{N[u_{1}], N[u_{2}]$, $\ldots, N[u_{t}]\}$ is a clique partition of $V(G)$. 
  Let $D$ be any minimal dominating set of $G\boxtimes H$. We argue that for each $i = 1, 2, \ldots, t$, the set
  \[	
    D_i=\left\{v\in V(H)\mid
      (u, v)\in D, u\in N_{G}[u_{i}]\right\}
  \]
  is a minimal dominating set of $H$.  Let $v$ be any vertex in $H$.  Since
  \[
    N_{G\boxtimes H}[(u_i, v)]=\left\{(u, v')\mid u\in N_{G}[u_{i}], v'\in N_{H}[v]\right\},
  \]
  we have $D_i\cap N_{H}[v]\neq \emptyset$.
  Thus, $D_i$ is a dominating set of $H$.  For the minimality, suppose for contradiction that there exists some vertex $v\in D_i$ such that $D_i\setminus \left\{v\right\}$ is still a dominating set of $H$.  Then $D\setminus \left\{(u, v)\mid u\in N_{G}[u_{i}]\right\}$ is a dominating set of $G\boxtimes H$, contradicting that $D$ is minimal.  
  
  It follows from the minimality of $D$ that for any $v\in V(H)$ and $i\in\left\{1, 2, \ldots, t\right\}$, $|D\cap (N_{G}[u_{i}]\times \{v\})|\leq 1$.  Thus for any $i\in\left\{1, 2, \ldots, t\right\}$,
  \[
    | D\cap (N_{G}[u_{i}]\times V(H)) |=|D_i|=\gamma(H).
  \] 
  Hence $|D|=t\gamma(H)$. Now that all the minimal dominating sets of $G\boxtimes H$ have the same size, $G\boxtimes H$ is well-dominated. 
\end{proof}

The proof implies a characterization of all minimal dominating sets of well-dominated strong product $G \boxtimes H$ when one factor is a trivially well-dominated graph.

According to
Prisner et al.~\cite{prisner1996well}, all well-dominated chordal graphs are trivially well-dominated.  Thus, they satisfy Theorem~\ref{thm:well-dominated}.
Indeed, the vertices in the definition of trivially well-dominated graphs must be \emph{simplicial} vertices (whose closed neighborhood is a clique).
Note that a trivially well-dominated graph is not necessarily chordal, e.g., the graph obtained from a $C_{6}$ by adding two edges, between the first and the third vertices and between the fourth and the sixth vertices, respectively.

\begin{conjecture}
  Let~$G$ be a well-dominated graph.  If $G$ is not trivially well-dominated, then there exists another well-dominated graph~$H$ such that $G\boxtimes H$ is not well-dominated.
\end{conjecture}

We could not find a proof, but here are some observations.
Of connected graphs on at most four vertices, only $K_{1}$, $K_{4}$, $P_{4}$, and~$C_{4}$ are well-dominated, and the first three are trivially well-dominated.  The strong product graph $C_4\boxtimes C_4$ is not well-dominated because
\[
  \gamma(C_4\boxtimes C_4)=3 < \alpha(C_4\boxtimes C_4)=4.
\]
There are six connected well-dominated graphs on five vertices, and only three of them are trivially well-dominated.
For each of the others, the strong product with $C_{4}$ or $C_{5}$ is not well-dominated.

\section{Well-edge-dominated strong product graphs}

For a subset $S\subset V(G)$, we use $G- S$ to denote the subgraph of~$G$ after removing vertices in~$S$ from~$G$.
For a set~$X$ of edges, we use~$V(X)$ to denote the set of their endpoints.
We recall two simple facts on well-edge-dominated graphs, both of which
can be observed from that (1)~every maximal matching is a minimal edge dominating set; and (2)~any matching can be extended to a maximal matching.

\begin{theorem} [\cite{anderson2022well}] \label{le:con}
  Let $M$ be any matching of a graph $G$. If $G$ is well-edge-dominated (resp. equimatchable), then $G- V(M)$ is well-edge-dominated (resp. equimatchable). 
\end{theorem} 

\begin{proposition} [\cite{anderson2022well}] \label{th:per-mat}
  A connected well-edge-dominated graph contains a perfect matching if and only if it is a complete graph on four vertices or a complete bipartite graph with the same number of vertices in both parts.
\end{proposition}

By definition, a strong product graph has a perfect matching if either of its factors has.  Since the strong product of two connected nontrivial graphs contains a $K_{4}$, it cannot be a complete bipartite graph.  Thus, to prove Theorem~\ref{thm:well-edge-dominated}, it suffices to show that the strong product of two connected graphs without perfect matchings cannot be well-edge-dominated. 

\begin{lemma} \label{lem:chara}
  If the strong product of two nontrivial and connected graphs is well-edge-dominated, then at least one of the factors has a perfect matching.
\end{lemma}

We use $\alpha'(G)$ to denote the cardinality of a maximum matching of~$G$.  Recall that a graph has a perfect matching if $\alpha'(G) = |V(G)|/2$. 
We prove Lemma~\ref{lem:chara} by strong induction on $\alpha'(G)+\alpha'(H)$, where $G$ and $H$ are the factors.
Since both factors are connected, the base case is $\alpha'(G) = \alpha'(H) = 1$,
i.e., when both factors are either a triangle or a star (i.e., a complete bipartite graph $K_{1, n}$ for $n\geq 1$).

\begin{lemma} \label{le:mup}
  Let $G$ and $H$ be two connected graphs.
  If $\alpha'(G)=\alpha'(H)=1$, then
  $G\boxtimes H$ is well-edge-dominated if and only if both $G$ and $H$ are $K_2$.
\end{lemma}
\begin{proof}
  The sufficiency is obvious and we focus on the necessity.
  If one of $G$ and $H$ is $K_2$, then $G\boxtimes H$ has a perfect matching and the other must also be $K_2$ by Proposition~\ref{th:per-mat}.  Henceforth, the orders of both $G$ and $H$ are at least three, and we need to show that $G\boxtimes H$ is not well-edge-dominated.
  Note that $K_9$ is not well-edge-dominated because $\alpha'(K_9) = 4$ but it admits a minimal edge dominating set of size seven (any seven edges incident to a single vertex).  Thus, $G$ and $H$ cannot both be $K_{3}$.  Assume without loss of generality that $G$ is not $K_{3}$.
  Then $G$ is $K_{1, n}$ for some $n\geq 2$, and we denote its vertex set as $\{u_{0}, u_1, u_2, \ldots, u_n\}$, where $u_{0}$ is the vertex of degree~$n$.
  By virtue of Theorem~\ref{le:con}, it suffices to construct a matching~$M$ of~$G\boxtimes H$ and show that $(G\boxtimes H)- V(M)$ is not well-edge-dominated.
  
  If $H$ is $K_3$, with vertex set $\{v_1, v_2, v_3\}$, we use
  \[
    M= \{(u_{i}, v_1)(u_{i}, v_2)\mid 1 \le i \le n\}.
  \]
  To see that $(G\boxtimes H) - V(M)$ is not well-edge-dominated, note that $\{(u_0, v_1)(u_0, v_2),
  (u_0, v_3)( u_1, v_3)\}$ and $\{(u_0, v_1)(u_0, v_2), (u_0, v_1)(u_1, v_3), (u_0, v_1)(u_2, v_3),
  \ldots, (u_0, v_1)
  (u_n, v_3)\}$ are two minimal edge dominating sets of $(G\boxtimes H) - V(M)$.
  
  Now that $H$ is $K_{1, t}$ for some $t \geq 2$, let $V(H)=\left\{v_0, v_1, v_2, \ldots, v_t\right\}$, where $v_0$ is the vertex of degree $t$.
  We use the matching
  \[
    M = \{(u_0, v_i)(u_1, v_i)\mid 1 \le i \le t\} \cup \{(u_i, v_0)(u_i, v_1) \mid 3 \le i \le n\}.
  \]
  To see that $(G\boxtimes H) - V(M)$ is not well-edge-dominated, note that $\{(u_0, v_0)(u_2, v_0)\}$ and $\{(u_0, v_0)(u_1, v_0), (u_2, v_0)(u_2, v_1)\}$ are two minimal edge dominating sets of $(G\boxtimes H) - V(M)$. 
\end{proof} 

We are now ready for the main lemma.

\begin{proof}[Proof of Lemma~\ref{lem:chara}]
  Let~$G$ and~$H$ be the factors.
  We proceed by strong induction on $\alpha'(G)+ \alpha'(H)$.
  Since both graphs are connected, $\alpha'(G)+ \alpha'(H) \ge 1 + 1 = 2$.
  We have seen the base case, when $\alpha'(G)+ \alpha'(H)=2$, in Lemma~\ref{le:mup}.
  Now suppose $\alpha'(G)+ \alpha'(H)\ge 3$.
  Suppose for contradiction that $G\boxtimes H$ is well-edge-dominated, but neither $G$ nor $H$ has a perfect matching.
  Assume without loss of generality that $\alpha'(H)\geq 2$.
  We take a maximum matching~$M$ of~$H$, and an arbitrary edge~$e \in M$.
  Let~$A$ be the set of ends of edges in $M\setminus \{e\}$, and let $H^{\prime}$ be a largest component of~$H - A$.
  Note that $|A| = 2 \alpha'(H) - 2 > 0$.
  By Theorem~\ref{le:con}, both $G\boxtimes (H - A)$ and $G\boxtimes H'$ are well-edge-dominated, since $G\boxtimes H'$ is a connected component of $G\boxtimes (H - A)$.
  
  Since the edge~$e$ remains in~$H - A$, the order of $H^{\prime}$ is at least two.
  Since any matching of~$H'$ can be augmented to a matching of~$H$ by adding the edges in~$M\setminus \{e\}$, we have $\alpha'(H') < \alpha'(H)$. 
  Thus, $\alpha'(G)+\alpha'(H^{\prime})< \alpha'(G) + \alpha'(H)$, and $H'$ must have a perfect matching by the inductive hypothesis.
  But then $G\boxtimes H'$ has a perfect matching, and it has to be $K_4$ or $K_{n, n}, n\geq 1$, by Proposition~\ref{th:per-mat}.
  Since $|V(G)| > 2$, the order of~$G\boxtimes H'$ is at least six.  By definition, $G\boxtimes H'$ contains a triangle.  Thus, we end with a contradiction. 
\end{proof}

\section{Equimatchable strong product graphs}

The following is immediate from the definition of strong product graphs.
\begin{proposition}\label{lem:connectivity}
  The strong product of two nontrivial and connected graphs is 2-connected.
\end{proposition}
\begin{proof}
        Let $G$ and $H$ be the factors. It suffices to prove that for any pair of vertices $(u_1, v_1)$ and $(u_2, v_2)$ of $G\boxtimes H$, there are two vertex-disjoint paths between them. Since $G$ and $H$ are connected, let $P$ be one path from $u_1$ to $u_2$ in $G$, $Q$ be one path from $v_1$ to $v_2$ in $H$. Then $(P\times \{v_1\})\cup (\{u_2\}\times Q)$ and $(\{u_1\}\times Q)\cup (P\times \{v_2\})$ are two vertex-disjoint paths between $(u_1, v_1)$ and $(u_2, v_2)$ in $G\boxtimes H$.   		
\end{proof}

Lesk et al.~\cite{lesk1984equi} characterized equimatchable graphs that are 2-connected.
We use~$G- v$ as a shorthand for~$G- \{v\}$.
A graph~$G$ is \emph{factor-critical} if~$G- v$ has a perfect matching for every vertex~$v\in V(G)$.  

\begin{theorem}[\cite{lesk1984equi}] \label{le:2-connected}
  A 2-connected equimatchable graph is either factor-critical or bipartite or a complete graph of even order.
\end{theorem}	

By definition, the order of a factor-critical graph must be odd.  The following is immediate from Theorem~\ref{le:2-connected} and Proposition~\ref{lem:connectivity}.

\begin{corollary}\label{cr:equimatchable}
  Let~$G$ and~$H$ be two nontrivial and connected graphs.
  If~$G\boxtimes H$ is equimatchable, then
  \begin{description}
  \item[Even] it is complete when~$G\boxtimes H$ is even; or
  \item[Odd] it is factor-critical when~$G\boxtimes H$ is odd.
  \end{description}
\end{corollary}

This settles the even case of Theorem~\ref{thm:equimatchable}.
The rest is focused on the odd case.
The sufficiency follows from the observation of Akbari et al.~\cite{akbari2018equimatchable}.
\begin{proposition} [\cite{akbari2018equimatchable}] \label{le:odd order}
  Let $G$ be a nontrivial and connected graph of odd order. If $\alpha(G) \le 2$, then $G$ is equimatchable.
\end{proposition}

Recall that for a matching~$M$, the set of endpoints of all the edges in~$M$ is denoted as $V(M)$.
A matching~$M$ of a graph $G$ is \emph{near-perfect} if~$|V(M)| = |V(G)| - 1$.

The smallest odd, nontrivial, and connected graph are~$P_3$ and~$K_3$.
\begin{proposition}\label{lem:p3}
  Let~$G$ be an odd, nontrivial, and connected graph.  The strong product~$G\boxtimes P_3$ is equimatchable if and only if~$G$ is a complete graph.
\end{proposition}
\begin{proof}
  For the sufficiency, suppose that~$G$ is a complete graph.  Since~$\alpha(G\boxtimes P_3)=2$, the graph~$G\boxtimes P_3$ is equimatchable by Proposition~\ref{le:odd order}.
  
  For the necessity, suppose that~$G\boxtimes P_3$ is equimatchable.
  Let~$P_3$ be the path~$v_1 v_2 v_3$.
  We take a maximum independent set~$I$ of~$G$, and a maximal matching~$M_{G}$ of~$G - I$.
  Let
  \begin{align*}
    M_1 =& \{(u, v_1)(u, v_2)\mid u\in V(G)\setminus I\},
    \\
    M_2 =& \{(u, v_2)(u, v_3)\mid u\in I\},
    \\
    M_3 =& \{(u_{i}, v_3)(u_{j}, v_3)\mid u_{i} u_{j}\in M_{G}\}, 
  \end{align*}
  and $M = M_1\cup M_2\cup M_3$.
  By construction, $M$ is a matching of~$G\boxtimes P_3$.
  On the other hand, $M$ is maximal because
  \[
    V(G\boxtimes P_3)\setminus V(M) = I \times \{v_1\} \cup  (V(G)\setminus (V(M_{G})\cup I))\times \{v_{3}\}
  \]
  is an independent set.
  Since $G\boxtimes P_3$ is equimatchable and factor-critical (by Proposition~\ref{le:2-connected}), $|V(M)| = |V(G\boxtimes P_3)| - 1$.
  We have
  \[
    \alpha(G) = |I| = |I\times \{v_1\}| \le |V(G\boxtimes P_3)\setminus V(M)| = 1.
  \]
  This concludes the proof.
\end{proof}

\begin{proposition}\label{lem:K3}
  Let $G$ be an odd graph.
  If there exists an independent set $S$ of size three in $G$ such that $G - v$ has a perfect matching for any $v\in S$,
  then~$G\boxtimes K_3$ is not equimatchable.
\end{proposition}
\begin{proof}
  Let~$S = \{u_1, u_2, u_3\}$, and the three vertices in~$K_{3}$ are~$\{v_1, v_2, v_3\}$.
  By assumption, for $i = 1, 2, 3$, there is a perfect matching~$M_{i}$ of~$G - \{u_{i}\}$.
  Then
  \[
    M = \bigcup_{i=1}^{3}
    \{(u_{j}, v_{i})(u_{k}, v_{i})\mid u_{j} u_{k}\in M_{i}\}
  \]
  is a matching of~$G\boxtimes K_3$.
  It is maximal because the only three vertices not in~$V(M)$ are $\{(u_{i}, v_{i})\mid 1\le i \le 3\}$.
  Now that~$G\boxtimes K_3$ has a maximal matching that is not near-perfect, it cannot be equimatchable by Corollary~\ref{cr:equimatchable}.
\end{proof}

As one may expect, a necessary for a strong product to be equimatchable is that both factors are.  It is actually stronger: both have near-perfect matchings.

\begin{lemma} \label{lem:equimatchable}
  Let $G$ and $H$ be nontrivial and connected graphs of odd order. If $G\boxtimes H$ is equimatchable, then $G$ and $H$ are equimatchable and have near-perfect matchings.
\end{lemma}
\begin{proof}
  We show that~$G$ is equimatchable and has near-perfect matchings, and the other holds by symmetry.
  By Corollary~\ref{cr:equimatchable}, $G\boxtimes H$ is factor-critical, hence
  \[
    \alpha'(G\boxtimes H)= \frac{|V(G)||V(H)|-1}{2}.
  \]
  We take a maximum matching~$M_{G}$ of~$G$ and a maximum matching~$M_{H}$ of~$H$.  
  Let $M_{G}= \{u_{1} u'_{1}, u_{2} u'_{2}, \ldots, u_{\alpha'(G)} u'_{\alpha'(G)}\}$ and $M_{H}= \{v_{1} v'_{1}, v_{2} v'_{2}, \ldots, v_{\alpha'(H)} v'_{\alpha'(H)}\}$.
  We define 
  \begin{align*}
    M_1 &= \{(u, v_j)(u, v'_j)\mid u\in V(G), 1\leq j\leq \alpha'(H)\},
    \\
    M_2 &= \{(u_i, v)(u'_i, v)\mid 1\leq i\leq \alpha'(G), v\in V(H)\setminus V(M_H)\}.
  \end{align*}
  It is easy to check that~$M= M_1\cup M_2$ is a matching of~$G\boxtimes H$.
  Since~$M_G$ and~$M_H$ are maximum,~$V(G)\setminus V(M_G)$ and~$V(H)\setminus V(M_H)$ are independent sets in~$G$ and~$H$, respectively.
  By the definition of~$G\boxtimes H$,
  \[
    V(G\boxtimes H)\setminus V(M) = (V(G)\setminus V(M_G))\times (V(H)\setminus V(M_H))
  \]
  is an independent set in~$G\boxtimes H$, and hence~$M$ is maximal.
  Since~$G\boxtimes H$ is equimatchable and factor-critical,~$|(V(G)\setminus V(M_G))\times (V(H)\setminus V(M_H))|=1$, Thus,
  \[
    |V(G)\setminus V(M_G)|=1,\quad |V(H)\setminus V(M_H)|=1.
  \]
  This verifies that~$G$ has a near-perfect matching.

  By Theorem~\ref{le:con}, $(G\boxtimes H)- V(M_1)$ is equimatchable.
  Note that
  \[
    (G\boxtimes H)- V(M_1) = G\boxtimes (V(H)\setminus V(M_{H})),
  \]
  and it is isomorphic to~$G$ because $V(H)\setminus V(M_{H})$ is trivial.
  Thus, $G$ is equimatchable.  This concludes the proof.
\end{proof}

We prove the odd case of Theorem~\ref{thm:equimatchable} in two steps.  First, at least one factor is complete.  Second, the independence number of the other is at most two.
They are using Propositions~\ref{lem:p3} and~\ref{lem:K3}, respectively.
Note that if a connected graph is not complete if and only if it contains an induced~$P_{3}$.
Combining Theorem~\ref{le:con} and Proposition~\ref{lem:p3}, we can show one side is complete.
For the second step, we need the following technical lemma.

\begin{lemma} \label{le:perfect}
  Let $G$ be an odd connected graph that is equimatchable and has a near-perfect matching.  If~$\alpha(G) > 2$, then there exists an independent set $S$ of size three in $G$ such that $G - v$ has a perfect matching for any $v\in S$. 
\end{lemma}

Before presenting the proof of Lemma~\ref{le:perfect}, we use it to prove the main result of this section.

\begin{proof}[Proof of Theorem~\ref{thm:equimatchable}]
  We have seen the even case in Corollary~\ref{cr:equimatchable}.
  The proof is focused on the odd case.
  The sufficiency follows from Proposition~\ref{le:odd order}.
  For the necessity, suppose that~$G\boxtimes H$ is equimatchable.
  Recall that~$G$ and~$H$ are both odd, nontrivial, and connected graphs.

  In the first step, we show that at least one of~$G$ and~$H$ is a complete graph.
  We have nothing to show if~$G$ is a complete graph.  Hence, we assume otherwise and show that~$H$ must be a complete graph.
  Let~$M$ be a near-perfect matching of~$G$, which exists by Lemma~\ref{lem:equimatchable}, and let~$u_{0}$ be the unique vertex in~$V(G)\setminus V(M)$.
  We may assume that~$u_{0}$ is not universal (i.e., adjacent to all the other vertices in~$G$).
  Otherwise, we can replace any non-universal vertex (which exists since~$G$ is not a complete graph) with~$u_{0}$ in~$M$.
  Let~$u_{1} u_{2}$ be an edge in~$M$ such that~$u_{1}\not\in N(u_{0})$, and
  let~$U = \{u_{0}, u_{1}, u_{2}\}$.
  If~$V(G) \ne U$, then~$(G - U)\boxtimes H$ is even and has a perfect matching.
  Thus, $G[U]\boxtimes H$ is equimatchable by Theorem~\ref{le:con}.
  If $u_{0}$ and~$u_{2}$ are adjacent, then~$H$ is complete by Proposition~\ref{lem:p3}.
  Otherwise, $G[U]\boxtimes H$ is disconnected, and one component is~$G[\{u_{1}, u_{2}\}]\boxtimes H$.
  Since this component is even and equimatchable, it follows from Corollary~\ref{cr:equimatchable} that~$H$ is complete.

  In the second step, we may assume that~$H$ is complete (we can switch them otherwise) and we show $\alpha(G)\leq 2$.
  Suppose for contradiction that $\alpha(G) > 2$.
  According to Lemmas~\ref{lem:equimatchable} and~\ref{le:perfect}, $G$ contains an independent set~$S$ of size three such that~$G - \{u\}$ has a perfect matching for all~$u\in S$.
  Thus,~$G\boxtimes K_{3}$ is not equimatchable by Proposition~\ref{lem:K3}, and then~$G\boxtimes H$ is not equimatchable by Theorem~\ref{le:con}.
\end{proof}

For Lemma~\ref{le:perfect}, we need the Gallai--Edmonds decomposition theorem.  
Let $G$ be a graph, we partition $V(G)$ into~$A(G)\uplus C(G)\uplus D(G)$ defined as follows:
\begin{align*}
  B(G) =& \bigcap\, \{V(M)\mid M \text{ is a maximum matching}\},
  \\
  D(G) =& V(G)\setminus B(G),
  \\
  A(G) =& B(G) \cap N(D(G)),
  \\
  C(G) =& B(G) \setminus A(G).
\end{align*}

\begin{theorem}[Gallai--Edmonds Decomposition \cite{lovasz-plummer-86}] \label{th:gallai}
  Let the partition~$A(G)\uplus C(G)\uplus D(G)$ be defined as above.
  \begin{enumerate}[i)]
  \item Each component of the subgraph induced by~$D(G)$ is factor-critical.
  \item Every maximum matching of $G$ matches every vertex of $A(G)$ to a vertex of a distinct component of the subgraph induced by~$D(G)$.
  \end{enumerate}
\end{theorem}

\begin{lemma}[\cite{lesk1984equi}] \label{le:independent}
  Let $G$ be an odd connected graph.  If~$G$ is equimatchable, then $C(G)= \emptyset$ and $A(G)$ is an independent set of $G$. 
\end{lemma}

We are now ready for Lemma~\ref{le:perfect}.

\begin{proof}[Proof of Lemma~\ref{le:perfect}]
  Let~$A(G)\uplus C(G)\uplus D(G)$ be the Gallai--Edmonds decomposition of $G$.  By Lemma~\ref{le:independent}, $C(G)= \emptyset$ and $A(G)$ is an independent set.
  Let~$c$ denote the number of components in the subgraph induced by~$D(G)$.
  We claim that
  \[
    c = |A(G)| + 1.
  \]
  On the one hand, $c\le |A(G)| + 1$ since~$G$ has a near-perfect matching.
  On the other hand, $c\ge |A(G)|$ by the definition of~$A(G)$ and Theorem~\ref{th:gallai}(ii).
  Then the claim follows from that~$G$ is odd and all the components of the subgraph are odd by Theorem~\ref{th:gallai}(i).

  By definition, for each vertex~$x\in D(G)$, there is a maximum matching~$M$ of~$G$ such that~$x\in V(G)\setminus V(M)$.
  Since~$G$ has a near-perfect matching,~$M$ must be a perfect matching.
  Thus, for each vertex~$x\in D(G)$, the subgraph~$G - x$ has a perfect matching.
  To finish the proof, it suffices to find three pairwise nonadjacent vertices in~$D(G)$.
  It is trivial when~$c > 2$: we can take a vertex from each of the components of~$G - A(G)$.
  Note that~$c = 1$ if and only if~$A(G)$ is empty, when it follows from the assumption that~$\alpha(G) \ge 3$.

  In the remainder of the proof,~$c=2$.
  Let~$u_{0}$ be the vertex in~$A(G)$, and let~$C_1$ and~$C_2$ be the two components of~$G - u_{0}$.
    We claim that~$u_{0}$ is adjacent to all the vertices in at least one of them.
  Suppose otherwise, and for~$i=1,2$, let~$u_{i}$ be a vertex in~$C_{i}$ nonadjacent to~$u_{0}$.
  For~$i=1,2$, we take a perfect matching~$M_{i}$ of $C_{i} - \{u_{i}\}$.
  Since~$V(G)\setminus (V(M_1)\cup V(M_2)) = \{u_{0}, u_{1}, u_{2}\}$ is an independent set, $M_1\cup M_2$ is a maximal matching of $G$, contradicting the assumption.
  We may assume without loss of generality that~$u_{0}$ is adjacent to all the vertices in~$C_1$.
  Let~$S$ be a maximum independent set of~$G$; note that $|S| \ge 3$ by assumption.
  If~$u_{0}\in S$, we can replace~$u_{0}$ with any vertex from~$C_{1}$.  Thus, we have an independent set of three vertices of~$G - u_{0}$.  This concludes the proof.
\end{proof}

\bibliographystyle{plainurl}
\bibliography{ref}

\end{document}